\documentclass[12pt]{amsart}
\usepackage{hyperref}
\usepackage{amssymb}
\usepackage{amsmath}
\usepackage{amsthm}
\usepackage{graphicx}
\usepackage{color}
\usepackage{listings}
\usepackage{subfigure,multirow}
\usepackage{url,breakurl,enumerate,cite,rotating}
\usepackage[foot]{amsaddr}
\newcommand\shade{0.9}
\definecolor{lightgrey}{rgb}{\shade,\shade,\shade}

\newtheorem{theorem}{Theorem}
\newtheorem{corollary}{Corollary}

\newcommand{\cref}[1]{Corollary~\ref{cor:#1}}

\newcommand{\T}{{\textsf{T}}}

\newcommand\set[1]{\{ #1\}}
\newcommand\flr[2]{\left\lfloor \frac{#1}{#2} \right\rfloor}

\newcommand{\Tn}[1]{\mathbf T_{#1}}
\newcommand{\vd}{\mathbf V}
\newcommand{\hd}{\mathbf H}
\newcommand{\sd}{\mathbf S}
\newcommand{\vh}{\mathbf{V\hspace{-5pt}H}}

\newcommand{\bigO}{O}

\newcommand\vdlista[3]{$\mathtt{((\set{#1},\set{#2}),\set{#3})}$}
\newcommand\vdlistb[3]{$\mathtt{(\set{#1},(\set{#2},\set{#3}))}$}
\newcommand\vdlistc[2]{$\mathtt{(\set{#1},\set{#2})}$}

\begin{document}
\pagestyle{plain}

\lstset{language=Python,%
  basicstyle=\small\ttfamily,%
  keywordstyle=\color{blue}\bfseries,%
  commentstyle=\color{gray},%
  stringstyle=\color{purple}\ttfamily,%
  showstringspaces=false,%
}

\title{Generating Tatami Coverings Efficiently}%

\author{Alejandro Erickson}%
\email{alejandro.erickson@gmail.com}%
\address{School of Engineering and Computing Sciences, Durham
  University, DH1 3LE, United Kingdom}

 \author{Frank Ruskey}

 \address{Department of Computer Science, University of Victoria, V8W
   3P6, Canada}%

\begin{abstract}
  We present two algorithms to list certain classes of monomino-domino
  coverings which conform to the \emph{tatami} restriction; no four
  tiles meet.  Our methods exploit structural features of tatami
  coverings in order to create the lists in $\bigO(1)$ time per
  covering.  This is faster than known methods for generating certain
  classes of matchings in bipartite graphs.

  We discuss tatami coverings of $n\times n$ grids with $n$ monominoes
  and $v$ vertical dominoes, as well as tatami coverings of a two-way
  infinitely-wide strip of constant height, subject to the constraint that
  they have a finite number
  of non-trivial structural ``features''.

  These two classes are representative of two differing structural
  characterisations of tatami coverings which may be adapted to count
  other classes of tatami coverings or locally restricted matchings,
  such as tatami coverings of rectangles.
 \end{abstract}

 \maketitle

 \section{Introduction}

 The counting of domino coverings, together with its extension to
 counting perfect matchings in (planar) graphs, is a classic area of
 enumerative combinatorics and theoretical computer science
 (e.g.,~\cite{Stanley1985,Jerrum1990,TemperleyFisher1961,Kasteleyn1961,Kasteleyn1967,Valiant1979}).
 Less attention has been paid, however, to problems where the local
 interactions of the dominoes are restricted in some fashion
 (e.g.,~\cite{ChurchleyHuangZhu2011}).  Perhaps the most natural such
 restriction is the ``tatami'' condition, defined below.  The tatami
 condition is quite restrictive: for example, the $10\times 13$ grid
 cannot be covered with dominoes and also satisfy the tatami
 condition.  In this paper we discuss how to efficiently list certain
 types of tatami coverings.

Tatami mats are a traditional Japanese floor covering whose dimensions
are approximately $1\textup{m} \times 1\textup{m}$ or
$1\textup{m}\times
2\textup{m}$. 
In certain arrangements, no four tatami mats may meet.  A \emph{tatami
  covering} is a non-overlapping arrangement of $1\times 1$
monominoes, $1\times 2$ horizontal dominoes, and $2\times 1$ vertical
dominoes, in which no four tiles meet.  This local restriction forces
a global structure which is characterised in
\cite{EricksonRuskeySchurch2011} and \cite{EricksonSchurch2012}.

Takeaki Uno (in \cite{Uno1997}) generates various lists of matchings
for bipartite graphs, $G=(V,E)$, with time complexities of
$\bigO(|V|)$ per matching.  Our problems, which pertain to listing
locally restricted matchings on grid-graphs, are $\bigO(1)$ per
matching.  Specifically, we use the tatami structure to describe two
classes of tatami coverings which can be generated exhaustively in
constant amortised time (CAT), and discuss some general approaches to
generating more complex classes of tatami coverings.

The first of these classes are the tatami coverings of the $n\times n$ grid
with exactly $n$ monominoes, in a certain orientation, and $v$
vertical dominoes, denoted $\vd(n,k)$.  This is an extremal
configuration of tatami covering because $n$ is the maximum possible
number of monominoes in an $n\times n$ tatami covering, and it can be
described with a proper subset of the general tatami structure.

Let $\sd(n,k)$ be the subsets of $\set{1,2,\ldots, n}$ whose elements
sum to $k$.  Lemma 4.1 of \cite{EricksonRuskey2013a} describes a
bijection between the coverings in $\vd(n,k)$ and a union of sets of the
form $\sd(a,b)\times \sd(c,d)$ which is given in expressions~(\ref{eq:vh}--\ref{eq:vhx}).

Our technique for finding this bijection employs an operation which
preserves the tatami condition, called the \emph{diagonal flip},
defined in \cite{EricksonSchurch2012}.  The added observation that a
diagonal flip changes the orientation of the dominoes in it, enables us count
the coverings classified by the number of dominoes of each orientation.

The crux of the argument uses a partition of the $n\times n$ coverings
with $n$ monominoes which reveals diagonal flips each with
$1,2,\ldots, k$ dominoes, respectively, that can be flipped
independently.  We use this to express parts of a given tatami
covering as the number of subsets of $\{1,2,\ldots,n\}$ whose elements
sum to $k$.  An algorithm from \cite{BaronaigienRuskey1993} generates
the $k$-sum subsets of the $n$ set in constant amortised time, and we
use it to generate $\vd(n,k)$.

The second class comprises tatami coverings of a two-way infinitely
wide strip of constant height which have a finite number of
non-trivial structural features (first introduced in \cite{EricksonRuskeySchurch2011}).
This is a step toward generating all tatami coverings of rectangular
grids by obviating the geometrical difficulties of packing the
structural features into the rectangle.  The algorithm itself is quite
simple, arising from a system of homogeneous equations with positive
coefficients, however, the equations themselves are interesting
because they exert the power of the tatami structure.  We contrast
this with Theorem 3 of \cite{EricksonRuskeySchurch2011}, which gives
an algorithm for building a system of homogeneous linear recursive
equations to enumerate all tatami coverings of fixed-height
rectangles. Our equations, on the other hand, include the height of
the strip as a parameter.  We expect that our method for enumerating
strip coverings of fixed-height will serve to improve methods for
enumerating rectangular tatami coverings of fixed-height.







\section{Tatami coverings of square regions}
\label{sec:combalg}
\newcommand{\cc}{\texttt{C4}}%
\newcommand{\modcc}{\texttt{modC}}%
\newcommand{\Sset}{\mathbf S}

Let $\Tn{n}$ be the $n\times n$ coverings which are distinct under
rotation.  By Corollary 2 in \cite{EricksonRuskeySchurch2011}, we may
assume that these have monominoes in their top corners; all others being
obtained by rotation. Our goal is to
generate the coverings in $\Tn{n}$ which have exactly $k$ vertical
dominoes.


Let $\vd(n,k)$ and $\hd(n,k)$ be the coverings in $\Tn{n}$ with
exactly $k$ vertical and horizontal dominoes, respectively, and let
$\sd(n,k)$ be the subsets of $\set{1,2,\ldots, n}$ whose elements sum
to $k$.  We build our exhaustive generation algorithm from the
following bijection given in Lemma 4.1 of \cite{Erickson2013a}:

If $n$ is even, then $|\vd(n,k)|$ is equal to $|\vh(n,k)|$, which is
defined as the union of the sets
  \begin{align}
    \label{eq:vh}
    \bigcup_{i=1}^{\flr{n-1}{2}} \bigcup_{\substack{k_1+k_2 =\\k
        -(n-i-1)}}
    &\left(\left(\set{n-i-1}\times \sd(n-i-2,k_1)\right)\times \sd(i-1,k_2)\right.\\
    \notag    & \left. \cup \,\,\sd(i-1,k_2)\times \left(\set{n-i-1}\times \sd(n-i-2,k_1)\right)\right), \\
    \label{eq:vhx}
\text{ and } \bigcup_{k_1+k_2 = k}&\sd\left(\flr{n-2}{2},k_1\right)\times
    \sd\left(\flr{n-2}{2},k_2\right).
  \end{align}
When $n$ is odd, $|\vh(n,k)|$ is equal to $|\hd(n,k)|$.

Each integer in an element of $\vh(n,k)$ refers to a part of the
corresponding tatami covering, called a \emph{flipped diagonal}.  A
precise definition of this is in \cite{EricksonSchurch2012}, but referring
to the first two coverings in Figure~\ref{fig:diagonal} it should be clear
what a diagonal is, and what flipping it means.

Now consider the third diagram in this figure; it has four flipped
diagonals.  Furthermore, the largest of those flipped diagonals
contains 12 vertical dominoes, there are two others with the same
orientation, containing 3 and 8 vertical dominoes, and there is one
with the perpendicular orientation containing 1 vertical domino.  In
essence, this is what (\ref{eq:vh}) is counting; the outermost union
is over the number of dominoes, $n-i-1$, in the largest flipped
diagonal, and the inner union is over the other smaller flipped
diagonals, being careful that their respective sizes adds up to the
right total number of vertical dominoes $k$.

What, then, is the purpose of (\ref{eq:vhx})?  Note that a diagonal can flip up or
down.  If the largest flipped diagonal is the smaller of the two (up or down), then
a different classification arises, because there is no possibility of flipped
diagonals ``running into'' each other.  This leads to the slightly simpler expression
found in (\ref{eq:vhx}).  

The reader is encouraged to consult Figure \ref{fig:n8k7}.  It contains a complete
listing of $\vd(8,7)$ together with the corresponding sets.  The eight pairs from
(\ref{eq:vh}) are listed first, followed by eight elements from (\ref{eq:vhx}).
Note how the numbers match the number of dominoes in the various flipped diagonals.

\newcommand\factor{0.27\textwidth}
\newcommand\hspacefactor{\hspace{0.02\textwidth}}
\begin{figure}[h]
  \centering
\includegraphics[width=\factor]{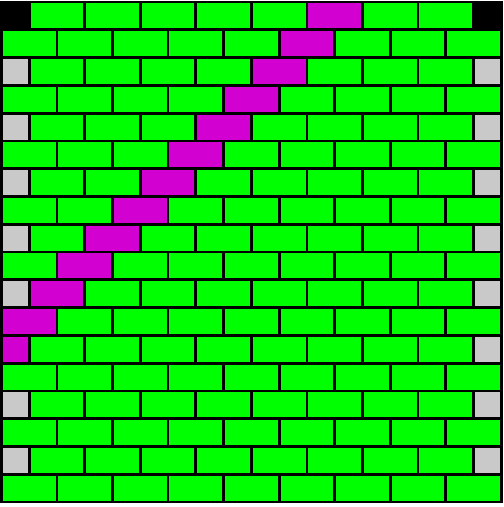} \hspacefactor %
\includegraphics[width=\factor]{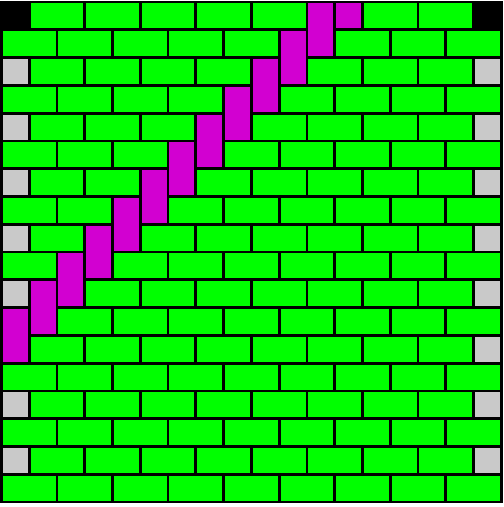} \hspacefactor %
\includegraphics[width=\factor]{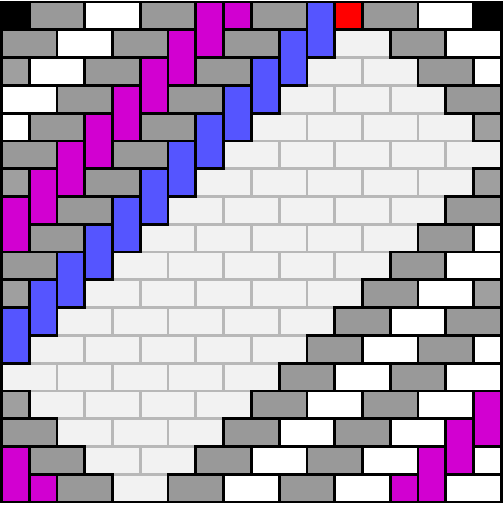} %
\caption{From left-to-right: an unflipped diagonal is shown in
  magenta, the diagonal is flipped, and the tatami covering
  corresponding to \vdlista{12}{3,8}{1}. This is an element of
  $(\set{n-5-1}\times \sd(n-5-2,11))\times \sd(5-1,1)$, where
  $n = 18$.}
\label{fig:diagonal}
\end{figure}


Our expression (\ref{eq:vh}--\ref{eq:vhx}) for $\vh(n,k)$ can be transformed into a CAT algorithm, provided that
we have a CAT algorithm for $\sd(n,k)$.  There is such a CAT algorithm, called
\cc~in~\cite{BaronaigienRuskey1993}.  One subtlety is that
invoking \cc$(n,k)$ requires $\Omega(n)$ preprocessing steps;
however, there are extreme values of $k$ for which $|\vh(n,k)| = o(n)$.
Since our goal is an algorithm that produces the (exponentially many) elements
of $\vh(n,k)$ in constant amortised time, and which uses only $O(n)$ time in
preprocessing, we will have to modify how the initialization is done.

The underlying data structure that is used by \cc~is an array
$\mathbf{c} = c[0],c[1],\ldots,c[n]$.  In order to represent a set
$\mathbf{a} = \{a_1 < a_2 < \cdots a_k \} \in \sd(n,k)$ we set $c[0] =
a_1$ and $c[a_i] = a_{i+1}$.  With these rules, the array $\mathbf c$ is like
a linked list of the elements of $\mathbf{a}$.  However, the unused
elements of $\mathbf c$ are also important, and here we set $c[i] = i+1$ for
all $i \not\in \mathbf{a}$, so that the initial array corresponds to
the correct values in the extreme cases.

A top level call to \cc$(n,k)$ takes the array $\mathbf{c_0} = [n+1,2,3,\ldots, n+1]$
as input (corresponding to the empty set), an array which takes $n+1$ steps to initialise.
However, when \cc$(a,i)$ concludes its computation, we also have that $\mathbf{c} = \mathbf{c}_0$.
Thus we can use most of $\mathbf{c}_0$ in initializing nearby successive calls to \cc.
In particular, we note that in (\ref{eq:vh}) successive values of $n$ vary by $\pm 1$,
which requires only a constant amount of change to the initialization.


Figure~\ref{fig:n8k7} shows the output of $\vd(8,7)$ of the algorithm
we have described.

\begin{figure}[h]
  \centering
\includegraphics[width=0.8\textwidth]{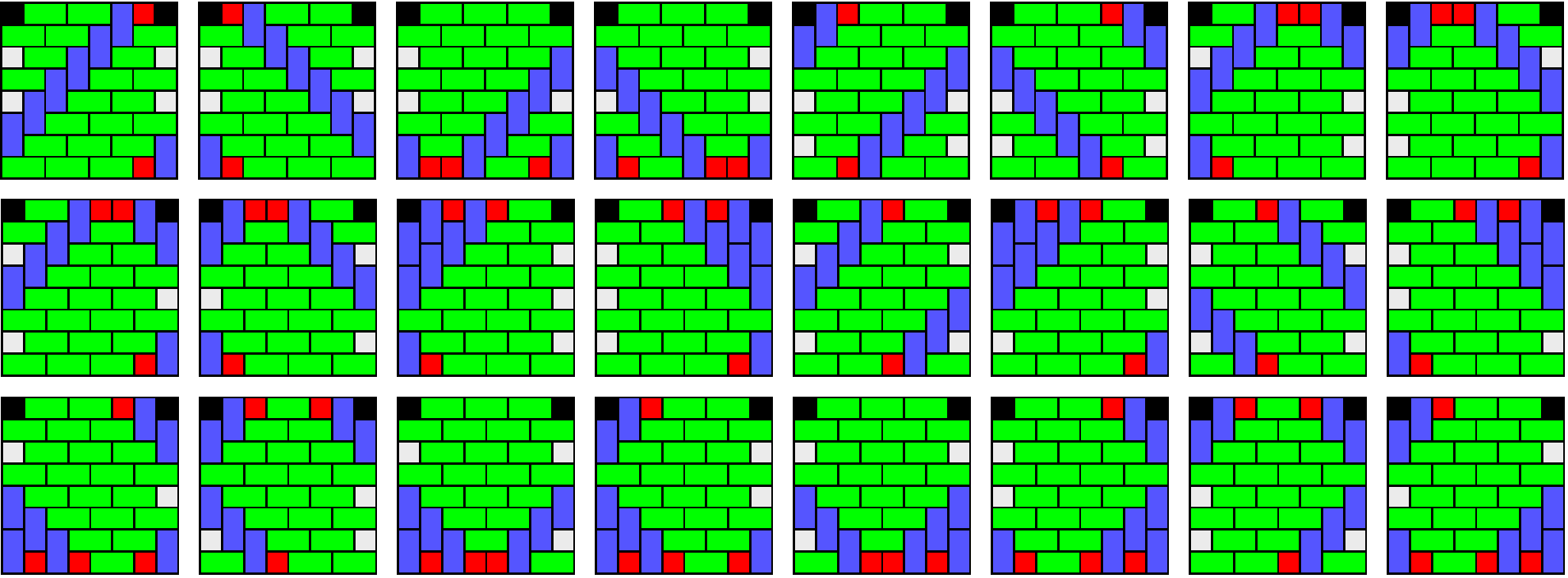}
\begin{sideways}
{\small
  \begin{tabular}{l}
\vdlista{6}{1}{}\\
\vdlistb{}{6}{1}\\
\vdlista{5}{1}{1}\\
\vdlistb{1}{5}{1}\\
\vdlista{5}{2}{}\\
\vdlistb{}{5}{2}\\
\vdlista{4}{}{2,1}\\
\vdlistb{2,1}{4}{}\\
\vdlista{4}{1}{2}\\
\vdlistb{2}{4}{1}\\
\vdlista{4}{2}{1}\\
\vdlistb{1}{4}{2}\\
\vdlista{4}{3}{}\\
\vdlista{4}{1,2}{}\\
\vdlistb{}{4}{3}\\
\vdlista{}{4}{1,2}\\
\vdlistc{1}{1,2,3}\\
\vdlistc{2}{2,3}\\
\vdlistc{3}{1,3}\\
\vdlistc{1,2}{1,3}\\
\vdlistc{1,3}{3}\\
\vdlistc{1,3}{1,2}\\
\vdlistc{2,3}{2}\\
\vdlistc{1,2,3}{1}
\end{tabular}
}
\end{sideways}
\caption{The coverings of $\vd(8,7)$ and their representations.}
\label{fig:n8k7}
\end{figure}

\begin{theorem}
  \label{thm:cattnk}
  There is a CAT algorithm which exhaustively generates all elements
  of $\vd(n,k)$.
\end{theorem}

\section{Finite tatami coverings of the infinite strip}
\label{sec:strip-covering}

The flipped diagonals of Section~\ref{sec:combalg} are a special case
of \emph{\T-diagrams}, given in
\cite{EricksonRuskeySchurch2011,Erickson2013a}, which characterise
monomino-domino tatami coverings of rectangles.  A \T-diagram is a
schematic of a tatami covering, which gives a set of structural
\emph{features} that determine the placement of the tiles.  A
\emph{strip} of height $r$ is a two-way infinitely wide integer grid
of constant height $r$.  The \T-diagrams, defined for rectangular
grids, also apply to the strip with the difference that there are no
vertical boundaries.  A \emph{finite monomino-domino tatami strip
  covering} is a monomino-domino tatami covering of the strip with a
finite number of \T-diagram features.  We refer to these as strip
coverings, as we do not consider any other type.

Two \T-diagram features are \emph{isomorphic}\index{isomorphic strip
  covering} if the respective sets of line segments they comprise are
horizontal translations of each other.  Two strip coverings are
isomorphic if their respective features, listed from left to right,
are isomorphic.  See Figure~\ref{fig:strip} for an overview by
example, or \cite{Erickson2013a} for a detailed description.

\begin{figure}[h]
  \centering
\includegraphics[width=0.8\textwidth]{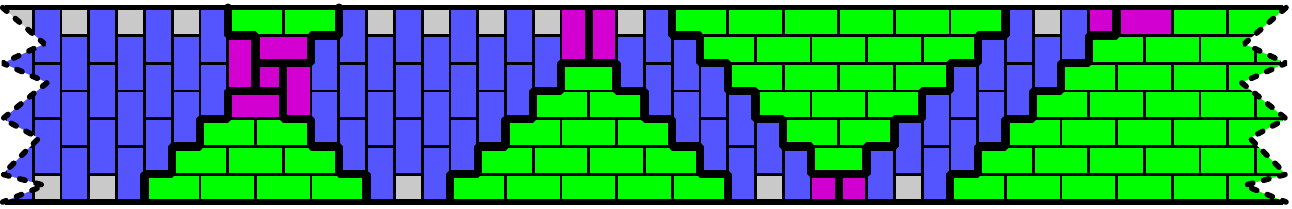}

\includegraphics[width=0.8\textwidth]{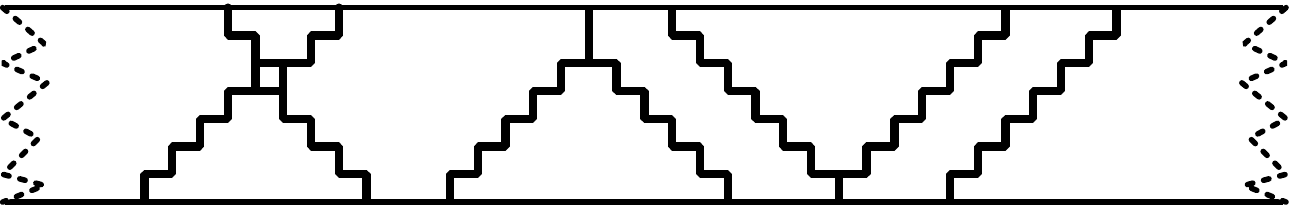}
  \caption{A tatami covering of the strip and its \T-diagram.  The
    features, from left to right, are a vortex, a bidimer, a vee, and a loner.}
  \label{fig:strip}
\end{figure}

Strip coverings, up to isomorphism, encapsulate some of the
combinatorial properties of rectangular tatami coverings without so
many of the geometric details that arise when packing feature diagrams
into a rectangle.  On the other hand, the \T-diagram of a strip
covering can be bounded by two vertical lines, thereby converting it
to a rectangular \T-diagram.  We show that there is a CAT algorithm to
generate the non-isomorphic height $r$ strip coverings with
$k$-features.

\begin{theorem}
  \label{thm:strip}
  If $R(r,n)$ is the number of non-isomorphic height $r$ strip
  coverings with exactly $n$ features, then it satisfies the system of
  homogeneous linear recurrence relations,
  \begin{align*}
    V_r(n) = & 4(r-1)V_r(n-1) + 2H_r(n-1),\\
    &\text{ where } V_r(0) = 1, V_r(1) = 4r-2;\\
    H_r(n) = & 2V_r(n-1), \text{ where } H_r(0) = 1;\\
    R(r,n) = & V_r(n) + H_r(n).
  \end{align*}
\end{theorem}
\begin{proof}
  Recall that a \T-diagram partitions the strip into regions, covered
  by vertical or horizontal \emph{bond}; that is, a rotation of the
  basic brick-laying pattern. Let $V_r(n)$ and $H_r(n)$ be the number
  of non-isomorphic strip coverings whose leftmost regions are
  vertical and horizontal bond, respectively.  The number of
  non-isomorphic features on the height $r$ strip are as follows:
  \begin{description}
  \item[Bidimers] There are $r-1$ vertical, and $r-1$ possible horizontal bidimers.
  \item[Vortices] There are $r-2$ clockwise, and $r-2$ possible
    counterclockwise vortices.
  \item[Vees] There is 1 vee on the top boundary and 1 vee on the bottom.
  \item[Loners] There are four loners, $\nearrow$, $\nwarrow$, $\searrow$, and $\swarrow$.  The first
    two occur on the bottom boundary, and the latter on top boundary.
  \end{description}
  All of the bidimers, vortices and vees have vertical bond to their
  left and right.  The $\searrow$ and $\nearrow$ loners have horizontal and vertical
  bond to their left and right, respectively, while the $\swarrow$ and $\nwarrow$
  loners have vertical and horizontal bond to their left and right,
  respectively.

  The bond coverings of the strip are either a horizontal or vertical
  bond.  These are counted by the initial conditions
  $H_r(0)=V_r(0)=1$.  If the leftmost region of the covering is
  horizontal bond, then the leftmost feature must be a $\searrow$ or
  $\nearrow$ loner.  The region to the left of the remaining features
  is a vertical bond, so $H_r(n) = 2V_r(n-1)$.

  The total number of features with vertical bond on their left side
  is $(r-1)+(r-1)+(r-2)+(r-2)+1+1+1+1$, so this gives $V_r(1)=4r-2$.
  Exactly two of these features, namely $\nwarrow$ and $\swarrow$
  loners, have horizontal bond on their right, so $V_r(n) =
  4(r-1)V_r(n-1) + 2H_r(n-1)$.

  Thus $R(r,n)=V_r(n)+H_r(n)$, as required.
\end{proof}

\begin{corollary}
  \label{cor:strip-gen}
  There exists a CAT algorithm for generating non-isomorphic, height $r$ strip
  coverings.
\end{corollary}
\begin{proof}
  There are $4r$ possible non-isomorphic features in height $r$ strip
  coverings, each of which can be expressed uniquely as an element of
  $\{0,1,\ldots , 4r-1\}$.  The recurrence relations in the proof of
  Theorem~\ref{thm:strip} describe a tree whose internal nodes are at
  least of degree $2$, and whose leaves all represent output.  The
  recursive algorithm which naturally arises from
  Theorem~\ref{thm:strip}, iterates through the features that can be
  added, given the bond of the leftmost region.  After adding each
  feature to the covering, using its unique symbol, the algorithm
  recurses.  There is a constant number of operations per call and a
  constant number of calls per leaf.  Therefore the algorithm is CAT,
  since there are more leaves than internal nodes.
\end{proof}

Generating strip coverings with \T-diagrams is a step towards doing
the same for tatami coverings of rectangles.  These are a natural
extension of domino fixed-height coverings, proposed by Knuth in
\cite{Knuth2011}, and discussed in~\cite{RuskeyWoodcock2009}.  Perhaps
one might proceed by considering the strip coverings whose \T-diagrams
are bounded on the left and right, and strip coverings with minimal
distance between features.  The desired respective positions for
adjacent pairs of features can be tabulated in the latter case in a
$4r\times 4r$ matrix.

Enumerating isomorphic strip coverings which fit in a bounded portion
of the strip, perhaps is equivalent to counting a type of integer
partition.  That is, the total amount of space between features is a
constant, while the placement of a feature can be shifted horizontally
by an even number of grid squares, if it is unhindered by a
neighbouring feature or a vertical boundary.



\bibliographystyle{plain}
\bibliography{bibliography}

\end{document}